\documentclass{article}

\usepackage[letterpaper,top=2cm,bottom=2cm,left=3cm,right=3cm,marginparwidth=1.75cm]{geometry}

\usepackage{amsmath}
\usepackage{amssymb}
\usepackage{amsthm}
\usepackage{graphicx}
\usepackage[colorlinks=true, allcolors=blue]{hyperref}
\usepackage{float}
\theoremstyle{plain}
\newtheorem{theorem}{Theorem}

\newtheorem{lemma}[theorem]{Lemma}

\theoremstyle{definition}

\title{The Singer--Zauner gap for equiangular tight frames}
\author{Matthew Fickus\thanks{Department of Mathematics and Statistics, Air Force Institute of Technology, Wright-Patterson AFB, OH} \and John Jasper\footnotemark[1] \and Dustin G.\ Mixon\thanks{Department of Mathematics, The Ohio State University, Columbus, OH} \thanks{Translational Data Analytics Institute, The Ohio State University, Columbus, OH}}
\date{}

\begin{document}
\maketitle

\begin{abstract}
We show that there does not exist a complex $d\times n$ equiangular tight frame with
\[
d^2-d+1<n<d^2.
\]
The proof, which originated from an internal model at OpenAI, mimics the relationship between real equiangular tight frames and strongly regular graphs.
\end{abstract}

\section{Introduction}

A complex \textit{equiangular tight frame} is a matrix $\boldsymbol{X}\in\mathbb{C}^{d\times n}$ for which there exist $\alpha,\beta>0$ such that
\[
|\boldsymbol{X}^*\boldsymbol{X}|=(1-\alpha)\boldsymbol{I}+\alpha\boldsymbol{J},
\qquad
\boldsymbol{X}\boldsymbol{X}^*
=\beta\boldsymbol{I},
\]
from which a trace-cycling argument gives
\[
\alpha=\sqrt{\frac{n-d}{d(n-1)}},
\qquad
\beta=\frac{n}{d}.
\]
Note that any equiangular tight frame is necessarily a \textit{unit norm tight frame}, meaning $\boldsymbol{X}$ has unit-norm columns and $\boldsymbol{X}\boldsymbol{X}^*
=\frac{n}{d}\boldsymbol{I}$.
Equiangular tight frames form optimal codes in projective space~\cite{Welch:74,CohnK:07}, and accordingly, they find a variety of applications~\cite{StrohmerH:03,HolmesP:04,RenesBSC:04,MixonQKF:13,BandeiraFMW:13}.
See~\cite{FickusM:15} for a survey.

We are interested in necessary conditions for equiangular tight frames.
Gerzon's bound~\cite{LemmensS:73} gives $n\leq d^2$, which Zauner~\cite{Zauner:99} conjectured is tight for every $d\geq2$.
A Naimark complement argument then gives another necessary condition: $n\leq (n-d)^2$.
Finally, the necessary condition $(d,n)\neq(3,8)$ follows from an expensive Gr\"obner basis calculation~\cite{Szollosi:14}.
These are the only necessary conditions that are known to date, though they are conjectured to be grossly insufficient~\cite{Mixon:15,Mixon:19}.

Much stronger necessary conditions are known for \textit{real} equiangular tight frames.
To see this, first apply an orthogonal transformation and sign the column vectors so that
\begin{equation}
\label{eq.block form}
\boldsymbol{X}
=\left[\begin{array}{cc}
1&\alpha\mathbf{1}^{\mathsf{T}}\\
\mathbf{0}&\sqrt{1-\alpha^2}\, \boldsymbol{Y}
\end{array}\right].
\end{equation}
Since $\boldsymbol{X}$ is a $d\times n$ unit norm tight frame, it follows that $\boldsymbol{Y}$ is a $(d-1)\times (n-1)$ unit norm tight frame with $\boldsymbol{Y}\boldsymbol{1}=\boldsymbol{0}$.
Furthermore, since $\boldsymbol{Y}$ is real, each off-diagonal entry of $\boldsymbol{Y}^*\boldsymbol{Y}$ takes one of two values, i.e., $\boldsymbol{Y}$ is a so-called \textit{two-distance tight frame}~\cite{BargGOY:15}.
By decomposing in terms of entry values, we get
\[
\boldsymbol{Y}^*\boldsymbol{Y}
=\boldsymbol{I}+\delta\boldsymbol{A}+\varepsilon\boldsymbol{B},
\]
where $\boldsymbol{A}$ and $\boldsymbol{B}$ are adjacency matrices of complementary strongly regular graphs.
In this way, real $d\times n$ equiangular tight frames correspond to $(v,k,\lambda,\mu)$-strongly regular graphs with $v=n-1$ and $k=2\mu$; see~\cite{Waldron:09}.
Thanks to this identification, one may collect strong necessary conditions from the combinatorial design literature~\cite{BrouwerVM:22}.

In this note, we mimic the real setting to obtain a new necessary condition on equiangular tight frames in the complex setting, which in turn recovers the condition $(d,n)\neq(3,8)$:

\begin{theorem}
\label{thm.singer-zauner gap}
There does not exist a complex $d\times n$ equiangular tight frame with $d^2-d+1<n<d^2$.
\end{theorem}

By a difference set construction due to Singer~\cite{Singer:38}, there exists a $d\times n$ equiangular tight frame with $n=d^2-d+1$ whenever $d-1$ is a prime power.
For this reason, we refer to Theorem~\ref{thm.singer-zauner gap} as the \textit{Singer--Zauner gap}.
The remainder of this paper proves Theorem~\ref{thm.singer-zauner gap} using an argument initially due to an internal model at OpenAI.
(N.B.:\ Despite the use of AI, this paper was written the old-fashioned way by the authors, who take full responsibility for the veracity of its claims.)

\section{Proof of the main result}

Suppose there exists a complex $d\times n$ equiangular tight frame with $n<d^2$.
We will show that $n\leq d^2-d+1$.
Since $\binom{d+1}{2}\leq d^2-d+1$, we may assume $n>\binom{d+1}{2}$ without loss of generality.
Apply a unitary transformation to the equiangular tight frame and phase the column vectors so that $\boldsymbol{X}$ takes the block form in~\eqref{eq.block form}.
Then
\[
\boldsymbol{X}^*\boldsymbol{X}
=\left[\begin{array}{cc}
1&\mathbf{0}^{\mathsf{T}}\\
\alpha\mathbf{1}&\sqrt{1-\alpha^2}\,\boldsymbol{Y}^*
\end{array}\right]
\left[\begin{array}{cc}
1&\alpha\mathbf{1}^{\mathsf{T}}\\
\mathbf{0}&\sqrt{1-\alpha^2}\,\boldsymbol{Y}
\end{array}\right]
=\left[\begin{array}{cc}
1&\alpha\mathbf{1}^{\mathsf{T}}\\
\alpha\mathbf{1}&\alpha^2\boldsymbol{J}+(1-\alpha^2)\boldsymbol{Y}^*\boldsymbol{Y}\end{array}\right].
\]
If $z$ is an off-diagonal entry of $\boldsymbol{H}:=\boldsymbol{Y}^*\boldsymbol{Y}$, then
\[
\alpha^2
=|\alpha^2+(1-\alpha^2)z|^2
=(1-\alpha^2)^2|z|^2+2\alpha^2(1-\alpha^2)\operatorname{Re}(z)+\alpha^4,
\]
which rearranges to
\[
|z|^2+2\gamma\operatorname{Re}(z)-\gamma=0,
\qquad
\gamma:=\frac{\alpha^2}{1-\alpha^2}.
\]
This suggests taking the entrywise product $\boldsymbol{K}:=\boldsymbol{H}\odot\overline{\boldsymbol{H}}$ and average $\boldsymbol{R}:=\frac{1}{2}(\boldsymbol{H}+\overline{\boldsymbol{H}})$ so that
\[
\boldsymbol{K}+2\gamma\boldsymbol{R}-\gamma\boldsymbol{J}
=(\gamma+1)\boldsymbol{I}.
\]
Indeed, the diagonals of $\boldsymbol{K}$, $\boldsymbol{R}$, $\boldsymbol{J}$, and $\boldsymbol{I}$ all equal $1$.
Strikingly, the real symmetric matrices $\boldsymbol{K}$ and $\boldsymbol{R}$ are simultaneously diagonalizable.
To see this, first note that $\boldsymbol{Y}\mathbf{1}=\mathbf{0}$ implies $\boldsymbol{R}\mathbf{1}
=\mathbf{0}$, and so the above identity gives $\boldsymbol{K}\mathbf{1}
=\frac{n-1}{d-1}\mathbf{1}$.
Next, every eigenvector of $\boldsymbol{K}$ that is perpendicular to $\mathbf{1}$ is also an eigenvector of $\boldsymbol{R}$ (and vice versa), and the corresponding eigenvalues $\kappa$ and $\rho$ satisfy the relation
\[
\kappa + 2\gamma\rho
=\gamma+1.
\]
In particular, we have $\operatorname{ker}(\boldsymbol{K})=\operatorname{ker}(\boldsymbol{R}-\mu\boldsymbol{I})$, where
\[
\mu
:=\frac{\gamma+1}{2\gamma}
=\frac{d(n-1)}{2(n-d)}.
\]
Furthermore, the following lemma gives that most eigenvalues and eigenspaces of $\boldsymbol{R}$ come in pairs.

\begin{lemma}
Given a complex matrix $\boldsymbol{P}$ such that $\boldsymbol{P}^2=\boldsymbol{P}$, let $\boldsymbol{A}$ and $\boldsymbol{B}$ be real matrices such that
\[
\boldsymbol{P}
=\boldsymbol{A}+i\boldsymbol{B}.
\]
For each $\lambda\in(0,1)$, it holds that multiplication by $\boldsymbol{B}$ defines a linear bijection
\[
\operatorname{ker}(\boldsymbol{A}-\lambda\boldsymbol{I})
\,\longrightarrow\,
\operatorname{ker}(\boldsymbol{A}-(1-\lambda)\boldsymbol{I}).
\]
\end{lemma}

\begin{proof}
Take the real and imaginary parts of the identity $\boldsymbol{P}^2=\boldsymbol{P}$ to get
\[
\boldsymbol{A}^2-\boldsymbol{B}^2
=\boldsymbol{A},
\qquad
\boldsymbol{A}\boldsymbol{B}+\boldsymbol{B}\boldsymbol{A}
=\boldsymbol{B}.
\]
Given $\lambda\in(0,1)$, then for every $\boldsymbol{x}\in\operatorname{ker}(\boldsymbol{A}-\lambda\boldsymbol{I})$, it holds that
\[
\boldsymbol{A}\boldsymbol{B}\boldsymbol{x}
=(\boldsymbol{B}-\boldsymbol{B}\boldsymbol{A})\boldsymbol{x}
=\boldsymbol{B}\boldsymbol{x}-\boldsymbol{B}(\boldsymbol{A}\boldsymbol{x})
=(1-\lambda)\boldsymbol{B}\boldsymbol{x},
\]
i.e., $\boldsymbol{B}\boldsymbol{x}\in\operatorname{ker}(\boldsymbol{A}-(1-\lambda)\boldsymbol{I})$.
Thus, for each $\lambda\in(0,1)$, multiplication by $\boldsymbol{B}$ defines a linear map
\[
L_\lambda\colon \operatorname{ker}(\boldsymbol{A}-\lambda\boldsymbol{I})\to\operatorname{ker}(\boldsymbol{A}-(1-\lambda)\boldsymbol{I}).
\]
To show $L_\lambda$ is injective, suppose $L_\lambda(\boldsymbol{x})=\boldsymbol{0}$.
Then
\[
\boldsymbol{0}
=\boldsymbol{B}L_\lambda(\boldsymbol{x})
=\boldsymbol{B}^2\boldsymbol{x}
=(\boldsymbol{A}^2-\boldsymbol{A})\boldsymbol{x}
=(\lambda^2-\lambda)\boldsymbol{x}
=-\lambda(1-\lambda)\boldsymbol{x}.
\]
Since $\lambda\notin\{0,1\}$, we have $\boldsymbol{x}=\boldsymbol{0}$.
Then $L_\lambda$ and $L_{1-\lambda}$ are both injective, so the domains of both maps have equal dimension, and the result then follows from rank--nullity.
\end{proof}

Since $\boldsymbol{Y}$ is a unit norm tight frame, $\boldsymbol{P}:=\frac{d-1}{n-1}\boldsymbol{H}$ is an orthogonal projection matrix.
Put
\[
\lambda
:=\frac{d-1}{n-1}\cdot\mu
=\frac{d(d-1)}{2(n-d)},
\]
and observe that $n<d^2$ implies
\[
\lambda
=\frac{d(d-1)}{2(n-d)}
>\frac{d(d-1)}{2(d^2-d)}
=\frac{1}{2},
\]
while $n>\binom{d+1}{2}$ implies
\[
\lambda
=\frac{d(d-1)}{2(n-d)}
<\frac{d(d-1)}{2(\binom{d+1}{2}-d)}
=1,
\]
i.e., $\lambda\in(\frac{1}{2},1)$.
By decomposing $\boldsymbol{P}=\boldsymbol{A}+i\boldsymbol{B}$ as in the lemma, it follows that
\[
\operatorname{ker}(\boldsymbol{K})
=\operatorname{ker}(\boldsymbol{R}-\mu\boldsymbol{I})
=\operatorname{ker}(\boldsymbol{A}-\lambda\boldsymbol{I})
\qquad
\text{and}
\qquad
\operatorname{ker}(\boldsymbol{R}-(\tfrac{n-1}{d-1}-\mu)\boldsymbol{I})
=\operatorname{ker}(\boldsymbol{A}-(1-\lambda)\boldsymbol{I})
\]
are orthogonal subspaces of $\operatorname{im}(\boldsymbol{R})$ of equal dimension.
Thus,
\[
\operatorname{rank}(\boldsymbol{R})
\geq\operatorname{dim}\Big(\operatorname{ker}(\boldsymbol{R}-\mu\boldsymbol{I})\oplus\operatorname{ker}(\boldsymbol{R}-(\tfrac{n-1}{d-1}-\mu)\boldsymbol{I})\Big)
=2\operatorname{nullity}(\boldsymbol{K}).
\]
Furthermore,
\[
\operatorname{rank}(\boldsymbol{K})
=\operatorname{rank}(\boldsymbol{H}\odot\overline{\boldsymbol{H}})
\leq\operatorname{rank}(\boldsymbol{H})\operatorname{rank}(\overline{\boldsymbol{H}})
=(d-1)^2,
\]
\[
\operatorname{rank}(\boldsymbol{R})
=\operatorname{rank}(\boldsymbol{H}+\overline{\boldsymbol{H}})
\leq\operatorname{rank}(\boldsymbol{H})+\operatorname{rank}(\overline{\boldsymbol{H}})
=2(d-1),
\]
and so
\[
2(d-1)
\geq\operatorname{rank}(\boldsymbol{R})
\geq2\operatorname{nullity}(\boldsymbol{K})
=2\big((n-1)-\operatorname{rank}(\boldsymbol{K})\big)
\geq2\big((n-1)-(d-1)^2\big),
\]
which rearranges to $n\leq d^2-d+1$.

\section*{Acknowledgments}

DGM thanks OpenAI for access to GPT-5.5 Pro, as well as an internal model.
JJ was supported in part by NSF DMS 2220320.
DGM was supported in part by NSF DMS 2220304.
The views expressed are those of the authors and do not reflect the official guidance or position of the United States Government, the Department of Defense, the United States Air Force, or the United States Space Force.

\end{document}